\documentclass[12pt]{amsart}
\usepackage{amssymb}
\usepackage{a4wide}
\usepackage{xcolor}
 \usepackage{boxedminipage}
\usepackage{enumerate, graphicx}
\def\Cscr{\mathcal C}

\def\Pscr{\mathcal P}
\def\Qscr{\mathcal Q}

\def\Var{{\mathfrak V}}
\def\Tscr{{\mathcal T}}
\def\Lscr{{\mathcal L}}

\def\Pscr{{\mathcal P}}
\def\M{{\mathcal M}}

\def\rank{{\mathrm{rank}}}
\def\cstar{{\mathrm{star}}}

  \newenvironment{note}[1][Note]
   {\bigskip\begin{center}\begin{boxedminipage}{4.5in}\setlength{\parindent}{1em}\noindent\textbf{#1. }}
   {\end{boxedminipage}\end{center}\bigskip}

  \makeatletter
  \CheckCommand*\refstepcounter[1]{\stepcounter{#1}%
      \protected@edef\@currentlabel
       {\csname p@#1\endcsname\csname the#1\endcsname}%
  }
  \renewcommand*\refstepcounter[1]{\stepcounter{#1}%
    \protected@edef\@currentlabel
      {\csname p@#1\expandafter\endcsname\csname the#1\endcsname}%
  }
  \def\labelformat#1{\expandafter\def\csname p@#1\endcsname##1}
  \DeclareRobustCommand\Ref[1]{\protected@edef\@tempa{\ref{#1}}%
     \expandafter\MakeUppercase\@tempa
  }
  \makeatother

  \makeatletter
  \newcommand{\numberlike}[2]{%
     \expandafter\def\csname c@#1\endcsname{%
         \expandafter\csname c@#2\endcsname}%
  }
  \makeatother

  \def\DefaultNumberTheoremWithin{section}

  \theoremstyle{plain}
  \newtheorem{Lemma}{Lemma}
     \numberwithin{Lemma}{\DefaultNumberTheoremWithin}
     \labelformat{Lemma}{Lemma~#1}
  
     \numberwithin{Claim}{\DefaultNumberTheoremWithin}
     \numberlike{Claim}{Lemma}
     \labelformat{Claim}{Claim~#1}

  \newtheorem{Theorem}{Theorem}
     \numberwithin{Theorem}{\DefaultNumberTheoremWithin}
     \numberlike{Theorem}{Lemma}
     \labelformat{Theorem}{Theorem~#1}
  \newtheorem{Corollary}{Corollary}
     \numberwithin{Corollary}{\DefaultNumberTheoremWithin}
     \numberlike{Corollary}{Lemma}
     \labelformat{Corollary}{Corollary~#1}
  \newtheorem{Proposition}{Proposition}
     \numberwithin{Proposition}{\DefaultNumberTheoremWithin}
     \numberlike{Proposition}{Lemma}
     \labelformat{Proposition}{Proposition~#1}
  
     \numberwithin{Conjecture}{\DefaultNumberTheoremWithin}
     \numberlike{Conjecture}{Lemma}
     \labelformat{Conjecture}{Conjecture~#1}
  
     \numberwithin{Situation}{\DefaultNumberTheoremWithin}
     \numberlike{Situation}{Lemma}
     \labelformat{Situation}{Situation~#1}
 
     \numberwithin{Note}{\DefaultNumberTheoremWithin}
     \numberlike{Note}{Lemma}
     \labelformat{Note}{Note~#1}
     
  \theoremstyle{definition}
  \newtheorem{Definition}{Definition}
     \numberwithin{Definition}{\DefaultNumberTheoremWithin}
     \numberlike{Definition}{Lemma}
     \labelformat{Definition}{Definition~#1}

  \theoremstyle{definition}
  
     \numberwithin{Question}{\DefaultNumberTheoremWithin}
     \numberlike{Question}{Lemma}
     \labelformat{Question}{Question~#1}

  \theoremstyle{definition}
  \newtheorem{Problem}{Problem}
     \numberwithin{Problem}{\DefaultNumberTheoremWithin}
     \numberlike{Problem}{Lemma}
     \labelformat{Problem}{Problem~#1}

     \theoremstyle{remark} \newtheorem{Remark}{Remark}
     \numberwithin{Remark}{\DefaultNumberTheoremWithin}
     \numberlike{Remark}{Lemma}
     \labelformat{Remark}{Remark~#1}
  \theoremstyle{remark}

  \newtheorem{Example}{Example}
     \numberwithin{Example}{\DefaultNumberTheoremWithin}
     \numberlike{Example}{Lemma}
     \labelformat{Example}{Example~#1}
  
     \labelformat{Case}{Case~#1}
     \numberwithin{Case}{Lemma}
  
     \labelformat{Step}{Step~#1}
     \numberwithin{Step}{Lemma}
     
   \newcommand{\new}[1]{\textcolor{black}{#1}}

\labelformat{section}{Section~#1}
  \labelformat{subsection}{Section~#1}
  \labelformat{subsubsection}{Section~#1}

\usepackage{tikz}

\title{The signed Varchenko Determinant for Complexes of Oriented Matroids}
\author{Winfried Hochst\"attler}
    \address{FernUniversit\"at in Hagen \\ 
          Fakult\"at f\"ur Mathematik und Informatik \\
          58084 Hagen\\
          Germany}
     \email{winfried.hochstaettler@fernuni-hagen.de}
     
\author{Sophia Keip}
    \address{FernUniversit\"at in Hagen \\ 
          Fakult\"at f\"ur Mathematik und Informatik \\
          58084 Hagen\\
          Germany}
     \email{sophia.keip@fernuni-hagen.de}

\author{Kolja Knauer}\address{Departament de Matem\`atiques i Inform\`atica,
Universitat de Barcelona, Spain}
\email{kolja.knauer@ub.edu}

\begin{document}
\begin{abstract}
  We generalize the (signed) Varchenko matrix of a hyperplane arrangement to
  complexes of oriented matroids and show that its determinant has a nice factorization. \new{This extends previous results on hyperplane arrangements and oriented matroids.}
\end{abstract}

\maketitle
\section{Introduction}
Varchenko~\cite{V} considered a symmetric matrix
which may be viewed as a bilinear form on the vector space of linear
forms of the set of regions of a hyperplane arrangement \new{$\mathcal{A}$} over some ordered
field $\mathbb{K}$.  The value of the product of the characteristic
vectors of regions $Q_i$ and $Q_j$ is given by a product $\prod_{e \in
  S(Q_i,Q_j)}w_e$, where the $w_e$ are weights on the hyperplanes
$H_e$ of the arrangements and $S(Q_i,Q_j)$ is the set of
hyperplanes that have to be crossed on a shortest path from $Q_i$ to
$Q_j$. \new{The corresponding {\em Varchenko
  Matrix} $B_{\mathcal{A}}$ has entries of the form $\prod_{e \in
  S(Q_i,Q_j)}w_e$ for any pair of regions $Q_i$ and $Q_j$, also see \ref{thm:varchenko}.} In order to determine when the bilinear form is degenerate,
Varchenko~\cite{V} gave an elegant factorization of the determinant of
that matrix, considering the weights as variables.

\begin{Theorem}[Varchenko 1993 \cite{V}]\label{thm:varchenko_original} Let $\mathcal A$ be a real hyperplane
  arrangement, \new{$B_{\mathcal{A}}$ its Varchenko matrix,} and $L(\mathcal{A})$ the \emph{geometric lattice} formed by the
  intersections of hyperplanes in $\mathcal{A}$, then 
\[ \det(B_{\mathcal{A}})= \prod_{F\in L(\mathcal{A})} \left(1-w_F^2\right)^{m_F}\]
where $w_F=\prod_{F \subset H_e} w_e$ and $m_F$ are positive integers depending only on $L(\mathcal{A})$.  
\end{Theorem}

After the original proof of Varchenko there were several approaches to
provide cleaner proofs of this result. Denham and
Henlon~\cite{DH} sketched an elegant alternative way to prove
the result. Gente~\cite{G} provided some more details for that
proof and  generalized their approach to cones, which are
also called topcones or in our notation supertopes, i.e.\ convex sets
of regions. This method was generalized by Hochst\"attler and
Welker~\cite{WH} to oriented matroids, which form a
combinatorial model for hyperplane arrangements reflecting their
local linear structure but allowing for some global non-linearities. 
Aguiar and Mahajan~\cite{AM} generalized the original proof of
Varchenko to a signed version of the matrix and also derived
the result for topcones. In the signed case one considers an oriented hyperplane arrangement and the entries of the signed Varchenko matrix depend on which side of a hyperplane a cell lies, see \ref{def:signed}.   Randriamaro~\cite{R1} generalized their proof to oriented matroids and proved the factorization formula also for topoplane arrangements \cite{R2}. In his habilitation thesis he furthermore showed very recently that topoplane arrangements form a topological representation for {\em complexes of
  oriented matroids} (COMs) \cite{R3}, leading to our result from a different perspective. 
Note that Varchenko and Brylawski~\cite{BV} already generalized another, simpler bilinear form from hyperplane arrangements to matroids.

Bandelt et al.\ generalized oriented matroids to COMs by relaxing the global symmetry while maintaining
convexity and local symmetry. This framework captures a variety of
classes beyond oriented matroids, e.g., distributive lattices,
CAT(0)-cube complexes, lopsided sets, linear extension graphs, and
affine oriented matroids, see~\cite{BCK}.
Hochstättler and Welker proved the factorization
  formula not only for the full oriented matroid but also for
  supertopes, i.e.\ topal fibers in oriented matroids.  Every
  supertope is a COM and it has even been conjectured in
  \cite[Conjecture 1]{BCK} and~\cite[Conjecture 1]{KM} that any COM
  can be realized as a supertope of an oriented matroid. Also see~\ref{prob:supertopes}. However,
  that conjecture seems out of reach at the moment.  A big part of the  motivation for the work in the present paper is to support that
  conjecture by explicitly proving the factorization formula for general 
  COMs.  The presentation as
well as the proof follow the lines of Hochst\"attler and
Welker~\cite{WH}.  We furthermore achieve a generalization to the signed version of the Varchenko matrix, thus generalizing
Randriamaro~\cite{R1}.

The paper is organized as follows. In \ref{sec:2} we introduce
the considered structures. In \ref{sec:3} we present some 
tools from algebraic topology that we need for the proof of the main
theorem. The latter is presented in \ref{sec:4}. \new{We give some examples and applications in~\ref{sec:applications} and conclude the paper with some further remarks in~\ref{sec:conclusions}.}

\section{The Varchenko Determinant and Complexes of Oriented Matroids} \label{sec:2}
Before we introduce the Varchenko Determinant, we need to get familiar
with COMs. COMs have been introduced
in~\cite{BCK} as a common generalization of oriented matroids, affine
oriented matroids, and lopsided sets. We will use the notation from \cite{BCK} and \cite{thebook}. Note that the symbols $+,-$ and $0$ act like $1,-1$ and $0$ when it comes to negation and multiplication. We start with the following
definitions and axioms.

\begin{Definition} 
\new{We consider \emph{sign vectors} on a finite \emph{ground set} $E$, i.e., elements of $\{0,+,-\}^E$.} The \emph{composition} of two sign vectors $X$ and $Y$ is defined as the sign-vector

$$(X \circ Y)_e = \begin{cases}
                        X_e & \text{ if } X_e \neq 0,\\
                        Y_e & \text{ if } X_e = 0\\
                       \end{cases} 
                    \forall e\in E.$$

\noindent \new{The \emph{reorientation} of $X$ with respect to $A\subseteq E$ is defined as the sign-vector}
$$_AX_e =  \begin{cases}
                        -X_e & \text{ if } e\in A,\\
                        X_e & \text{ if } e\notin A\\
                       \end{cases} 
                    \forall e\in E.$$

\noindent The \emph{separator} of $X$ and $Y$ is defined as
\begin{align*}
S(X,Y) = \{e \in E: X_e = -Y_e \neq 0\}.
\end{align*}
The \emph{support} of $X$ is defined as
\begin{align*}
\underline{X} = \{e \in E: X_e \neq  0\}.
\end{align*}
The \emph{zero-set} of $X$ is defined as
\begin{align*}
z(X) = E\backslash\underline{X} = \{e \in E: X_e =  0\}.
\end{align*}

\end{Definition}

\noindent \new{For a set $\mathcal{L} \subseteq \{0,+,-\}^E$ we introduce five axioms:}
\begin{description}
 
\item[(FS)]\emph{Face Symmetry} 
\begin{align*}
\forall X,Y \in \mathcal{L}: X \circ (-Y) \in \mathcal{L}.
\end{align*}
\item[(SE)]\emph{Strong Elimination} 
\begin{align*}
&\forall X,Y \in \mathcal{L}\, \forall e \in S(X,Y)\, \exists Z \in \mathcal{L}: \\
&Z_e=0 \text{ and }\forall f \in E \setminus S(X,Y): Z_f = (X \circ Y)_f.
\end{align*} 

\item[(C)]\emph{Composition} 
\begin{align*}
\forall X,Y \in \mathcal{L}: X \circ Y \in \mathcal{L}.
\end{align*}

\item[(Z)]\emph{Zero} 
\begin{align*}
\text{The all zeros vector } \mathbf{0} \in \mathcal{L}.
\end{align*}

\item[(Sym)]\emph{Symmetry} 
\begin{align*}
\forall X\in \mathcal{L}: -X\in \mathcal{L}.
\end{align*}

\end{description}
Now we can define the term COM.
\begin{Definition}[Complex of Oriented Matroids (COM)] Let $E$ be a finite set and $\mathcal{L} \subseteq \{0,+,-\}^E$. The pair $\M=(E,\mathcal{L})$ is called a COM, if $\mathcal{L}$ satisfies (FS) and (SE). The elements of $\mathcal{L}$ are called \emph{covectors}.
\end{Definition}

\new{Let us first present OMs as special COMs.}

\begin{Definition}[Oriented Matroid (OM)] Let $E$ be a finite set and $\mathcal{L} \subseteq \{0,+,-\}^E$. The pair $\M=(E,\mathcal{L})$ is called an OM, if it is a COM that satisfies (Z).
\end{Definition}

\begin{Remark}\label{rem:OMs}
\new{Usually OMs are defined satisfying (C),(Sym),(SE). But note that (FS) implies (C). Indeed, by (FS) we first get $X\circ -Y\in\mathcal{L}$ and then $X\circ Y= (X\circ -X)\circ Y= X\circ -(X\circ -Y) \in \mathcal{L}$ for all $X,Y\in \mathcal{L}$. Further, (Z) together with (FS) clearly implies (Sym). Conversely, (Sym) and (C) imply (FS) while (Sym) and (SE) imply (Z).}
\end{Remark}

Let $\M=(E,\mathcal{L})$ be a COM. In the following we assume that $\M=(E,\mathcal{L})$ is \emph{simple}, i.e.\ 
\[\forall e \in E: \{X_e \mid X \in \mathcal{L}\} = \{+,-,0\}\quad \text{ and }\quad \forall e \neq f \in E: \{X_e X_f \mid X \in \mathcal{L}\} = \{+,-,0\}.\] In this setting the sign-vectors in $\mathcal{L}$ of full support are called \emph{topes} and their collection is denoted by $\Tscr$. 

The \emph{restriction} of a sign-vector $X\in\{0,+,-\}^E$ to $E \backslash A$, $A \subseteq E$, denoted by $X\backslash A \in \{0,+,-\}^{E\backslash A}$, is defined by $(X \backslash A)_e = X_e$ for all $e \in E \backslash A$. We also write $X|_{E\backslash A}$. The \emph{deletion} of a COM is defined by $(E \backslash A, \mathcal{L} \backslash A)$, where $\mathcal{L} \backslash A = \{X\backslash A,\; X \in \mathcal{L}\}$, also written as $\mathcal{L}|_{E \backslash A}$.

Let $\M=(E,\mathcal{L})$ be a COM and $S^+,S^- \subseteq E$ such that there exists a tope  $T\in\Tscr$ where $S^+$ is a subset of the positive elements of $T$ and $S^-$ is part of the negative elements of $T$. The \emph{topal fiber} $\rho_{(S^+,S^-)}(\Lscr)$ is defined by the covectors $\{X \mid X\in \mathcal{L},  X_e=+ \,\forall\, e\in S^{+}, X_e=- \,\forall\, e\in S^{-}\}$. We denote by $\Tscr(S^+,S^-)$ the set of topes of $\rho_{(S^+,S^-)}(\Lscr)$.
Since all covectors of a topal fiber have the same entries on $S^+ \cup S^-$, we usually suppress the redundant coordinates in $S^+ \cup S^-$, to obtain a simple COM on the groundset $E \backslash (S^+ \cup S^-)$. We will make use of the fact (shown in~\cite{BCK}) that the class of simple COMs is closed under deletion and under taking topal fibers.

For a covector $X\in \mathcal{L}$, the set $F(X)=\{X\circ Y\mid Y\in\mathcal{L}\}$ is usually called the \emph{face} of $X$. We define $\cstar(X) = \{T \in \Tscr\,|\,X \leq T\}$, where the componentwise ordering with respect to $0<+,-$ is used. Note that $\cstar(X)\backslash
  \underline{X}$ is the set of topes of $(E\backslash
  \underline{X}, F(X)\backslash \underline{X})$, which is well-known and easily seen to be an oriented matroid.

Let us look at a special OM which we will need in the next chapter. 
\begin{Definition}[Graphic OM of a directed $n$-cycle]\label{def:cycle}
This OM has a ground set $E$ of size $n$ and its set of covectors $\mathcal{C}_n$ consists of $\mathbf{0}$ and all compositions of sign-vectors from $\{0,+,-\}^E$ with exactly one positive and exactly one negative entry. Those generating sign-vectors are called the cocircuits of $\mathcal{C}_n$.
\end{Definition}

It can easily be checked that $\mathcal{C}_n$ is the set of covectors of an OM. We use $\mathcal{C}_3$ as an example:


\begin{Example}[Graphic OM of a directed triangle] \label{ex:cycle} We look at a digraph with three vertices which just consists of a directed cycle, i.e.
 \begin{center}
 \includegraphics[scale=0.5]{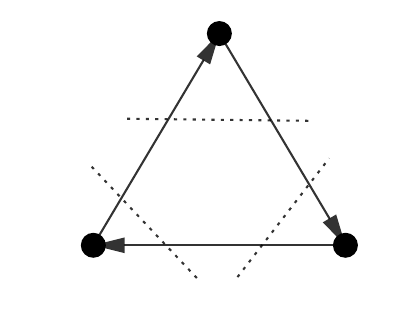}
 \end{center}
The ground set $E$ of this OM corresponds to the three arcs. One gets the covectors of such an OM by looking at the sign vectors of directed cuts (indicated with dotted lines). These sign vectors are $(+,-,0)$, $(-,+,0)$, $(+,0,-)$, $(-,0,+)$, $(0,+,-)$ and $(0,-,+)$. Their compositions additionally yield the covectors $(+,+,-)$, $(+,-,+)$, $(-,+,+)$, $(-,-,+)$, $(-,+,-)$ and $(+,-,-)$. We see that $\Tscr$ consists of all full support sign vectors, except $(+,+,+)$ and $(-,-,-)$.  

\end{Example}

We define the signed Varchenko matrix for COMs analogously to this matrix for hyperplane arrangements in \cite{AM}. For this purpose we introduce two variables $x_{e}^+$, $x_{e}^-$ for each element $e\in E$. Let $\mathbb{K}$ be a field and let $\mathbb{K}[x_{e}^*\,|\,*\in \{+,-\}, e\in E]$ the polynomial ring in the set of variables $x_{e}^*$, $*\in \{+,-\}$, $e\in E$.

\begin{Definition}[Signed Varchenko Matrix of a COM]\label{def:signed} Let $\M=(E,\mathcal{L})$ be a COM. The signed Varchenko matrix $\Var$ of $\mathcal{M}$ is defined by a $\# \Tscr \times \# \Tscr$-Matrix over 
\begin{align*}
\mathbb{K}[x_{e}^*\,|\,*\in \{+,-\}, e\in E].
\end{align*}
Its rows and columns are indexed by the topes $\Tscr$ in a fixed linear order. For $P,Q \in \Tscr$
\begin{align*}
\Var_{P,Q} = \prod_{e \in S(P,Q)} x_{e}^{P_e}.
\end{align*}
\end{Definition} 
Note that the diagonal entries $\Var_{P,P}$ of the matrix are equal to 1. \new{Let us illustrate this definition with the graphic OM of a directed triangle.}

\begin{Example}[continued]
\new{The signed Varchenko Matrix of the graphic OM of a directed triangle is
\begin{align*}
\Var = \left(\begin{array}{cccccc}
1&x_2^+ x_3^-&x_1^+ x_3^-&x_1^+ x_2^+ x_3^-&x_1^+&x_2^+\\
x_2^- x_3^+&1&x_1^+ x_2^-&x_1^+&x_1^+ x_2^- x_3^+&x_3^+\\
x_1^- x_3^+&x_1^- x_2^+&1&x_2^+&x_3^+&x_1^- x_2^+ x_3^+\\
x_1^- x_2^- x_3^+&x_1^-&x_2^-&1&x_2^- x_3^+&x_1^- x_3^+\\
x_1^- &x_1^- x_2^+ x_3^-&x_3^-&x_2^+ x_3^-&1&x_1^- x_2^+\\
x_2^-&x_3^-&x_1^+ x_2^- x_3^-&x_1^+ x_3^-&x_1^+ x_2^-&1
\end{array} \right)
\end{align*}}

\end{Example}

In this work we will prove the following theorem.

\begin{Theorem}
   \label{thm:varchenko}
   Let $\Var$ be the signed Varchenko matrix of the COM $\M=(E,\mathcal{L})$.
   Then
   \begin{align*}
     \det (\Var) = \prod_{Y \in \mathcal{L}} (1-a(Y))^{b_Y}. 
   \end{align*}
   where $a(Y) := \prod_{e \in z(Y)} x_{e}^+ x_{e}^-$ and $b_Y$ are nonnegative integers \new{that can be explicitly computed, see  \ref{note:bY}.}
\end{Theorem}

\begin{Example}[continued]
\new{For our example the determinant of the signed Varchenko matrix factorizes to
\begin{align*}
\det (\Var) = (1 - x_1^+ x_1^-)^2 (1 - x_2^+ x_2^-)^2 (1 - x_3^+ x_3^-)^2 (1 - x_1^+ x_1^- x_2^+ x_2^- x_3^+ x_3^-).
\end{align*}}
\end{Example}

A corollary of this result, namely the case where $x_{e}^-=x_{e}^+$, which is the original version of the Varchenko matrix, has been already proven for OMs in \cite{WH}. We formulate it for COMs.

\begin{Corollary}\label{cor:varchenko}
Let $\mathbf{V}$ be the (unsigned) Varchenko matrix (i.e. $x_{e}^-=x_{e}^+=x_e$) of the COM $\M=(E,\mathcal{L})$.
   Then
   \begin{align*}
     \det (\mathbf{V}) = \prod_{Y \in \mathcal{L}} (1-c(Y)^2)^{b_Y}. 
   \end{align*}
   where $c(Y) := \prod_{e \in z(Y)}x_{e}$ and $b_Y$ are nonnegative integers.
\end{Corollary}

\begin{Example}[continued]
\new{For our example the determinant of the (unsigned) Varchenko matrix factorizes to
\begin{align*}
\det (\Var) = (1 - x_1^2)^2 (1 - x_2^2)^2 (1 - x_3^2)^2 (1 - x_1^2 x_2^2 x_3^2).
\end{align*}}
\end{Example}

\section{Preparation}\label{sec:3}

We start with some basics about partially ordered sets $\mathcal{P}$ (\emph{posets}). For an introduction we recommend \cite{W}. One can associate an abstract simplicial complex $\Delta(\mathcal{P})$, called \emph{order complex}, to every poset. The elements of $\mathcal{P}$ are the vertices of this complex and the chains (i.e. totally ordered subsets) the faces. \new{Two posets are \emph{homotopy equivalent} if their order complexes are homotopy equivalent.} A poset is called \emph{contractible} if its order complex is homotopy equivalent to a point. Clearly a poset is contractible if it has a unique minimal or a unique maximal element, since this element is contained in every maximal chain and consequently in every maximal face of the order complex. For details see \cite{W}. We introduce now the \emph{M\"obius function} $\mu$ of a poset:
\begin{align*}
&\mu(x,x) = 1\text{ for all }x \in \mathcal{P}\\
&\mu(x,y) = -\sum_{x \leq z <y} \mu(x,z)\text{ for all }x <y \in \mathcal{P}.  
\end{align*}

The \emph{bounded extension} $\hat{\mathcal{P}}$ of a poset is the poset together with a new maximal element $\hat{1}$ and a new minimal element $\hat{0}$. 
The \emph{M\"obius number} of $\mathcal{P}$ is defined by

\begin{align*}
\mu(\mathcal{P}) = \mu(\hat{0},\hat{1}), 
\end{align*}
where the right-hand-side is evaluated in $\hat{\mathcal{P}}$.
\begin{Example}\label{ex:point}
Let us look at the poset $\mathcal{P}$ which consists only of one element. In the following its bounded extension and the value of the M\"obius function of the elements of the bounded extension are depicted.
\begin{center}
\includegraphics[scale=0.5]{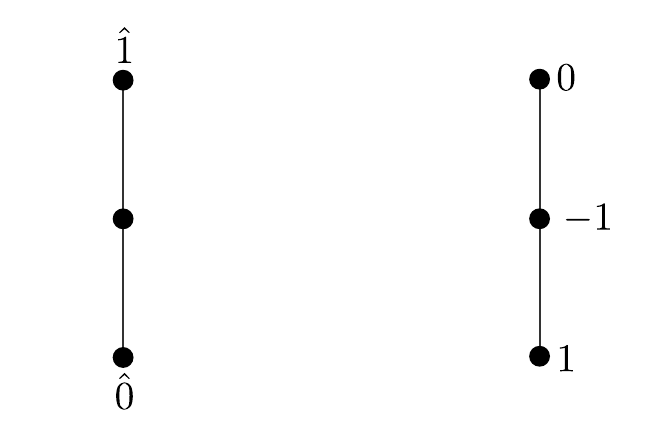}
\end{center}
Hence, the M\"obius number of the poset consisting of only one element is
\begin{align*}
\mu(\mathcal{P}) = \mu(\hat{0},\hat{1}) = 1 + (-1) = 0.
\end{align*}
\end{Example}

It follows from the following fact that the M\"obius number is a topological invariant with respect to homotopic equivalence.
\begin{Theorem}\cite[Philip Hall Theorem]{W} The M\"obius number of a poset equals the reduced Euler characteristic of its order complex, i.e.
\begin{align*}
\mu(\mathcal{P}) = \chi(\Delta(\mathcal{P}))-1.
\end{align*}
\end{Theorem}

In particular we get the following corollary, whose second part follows from the definition of contractability and \ref{ex:point}.

\begin{Corollary}
  \label{cor:moebius}
  For two homotopy equivalent posets $\mathcal{P}$ and $\mathcal{Q}$ we have
  $\mu(\mathcal{P}) = \mu(\mathcal{Q})$. In particular, if $\mathcal{P}$ is contractible then
  $\mu(\mathcal{P}) = 0$. 
\end{Corollary}

For a poset $\mathcal{P}$ and $p \in \mathcal{P}$ we denote by $\mathcal{P}_{\leq p}$ the subposet $\{ q \in \mathcal{P}~|~q \leq p\}$. 

\begin{Proposition}\cite[Quillen Fiber Lemma]{Q}
   \label{lem:quillen}
   Let $\Pscr$ and $\Qscr$ be posets and $f : \Pscr \rightarrow \Qscr$
   \new{order preserving}.  If for all $q \in Q$ we have that $f^{-1} (\Qscr_{\leq
     q})$ is contractible, then $\Pscr$ and $\Qscr$ are homotopy
   equivalent.
\end{Proposition}

 We will now associate posets with COMs, so let $\M=(E,\mathcal{L})$ be a COM and let
$R \in \{+,-\}^E$ be a fixed sign vector. We consider $\Tscr$ as a poset with order relation 
\[P \preceq_R Q\quad \text{ if }\quad S(R,P) \subseteq S(R,Q).\] We write $\Tscr_R$ if we consider
$\Tscr$ with this partial order and we call $R$ the \emph{base pattern} of the poset.

Now we will introduce a theorem which will help us with our crucial \new{\ref{cor:crucial}}.

\begin{Theorem}
  \label{thm:supertope}
Let $\M=(E,\mathcal{L})$ be a topal fiber of a COM $\M'=(E',\mathcal{L'})$, $R'\in \Tscr'$ a tope of $\M'$ and $R=R'_{|E}$ its restriction to $E$. Then the order complex of $\Tscr_{R}$ is contractible.

\end{Theorem}

\new{Note that the restriction $R$ in the statement of  \ref{thm:supertope} is not necessarily a tope of $\M$.} In order to apply the Quillen Fiber Lemma in the proof of \ref{thm:supertope}
we need the following lemma. 

\begin{Lemma}
  \label{lem:fiberisst} Let $f \in E$ and $R=\{+\}^E$. Let $\Tscr\setminus f$
  denote the set of topes of $\M \setminus f$ and
  $\Tscr\hspace{-.1cm}\setminus \hspace{-.1cm}f_{R\setminus f}$ the corresponding tope poset with base pattern $R\setminus \{f\}$. Consider the order-preserving map $\pi_{f}:\Tscr_R \to \Tscr\hspace{-.1cm}\setminus \hspace{-.1cm}f_{R\setminus f}$ given by restriction.  Let $Q \in \Tscr\hspace{-.1cm}\setminus \hspace{-.1cm}f$. Then
  \[\pi_{f}^{-1}((\Tscr\hspace{-.1cm}\setminus \hspace{-.1cm}f_{R\setminus f})_{\preceq Q})=\Tscr(Q^+,\emptyset).\]
\end{Lemma}
\begin{proof}
  Let $\tilde Q \in (\Tscr\setminus f_{R \backslash f})_{\preceq Q}$.  As $R\setminus f$ is all
  positive, we must have $\tilde Q^- \subseteq Q^-$ and hence $Q^+
   \subseteq \tilde Q^+$ implying $\pi_{f}^{-1}(\tilde Q)
  \subseteq \Tscr(Q^+,\emptyset)$.
If on the other
  hand  $\hat Q \in \Tscr(Q^+,\emptyset)$, then
  $\hat Q^- \subseteq Q^-\cup\{f\},$ $Q^+
  \subseteq \hat Q^+\cup \{f\}$. Hence  $\pi_f(\hat Q) \preceq_{R\setminus f} Q$.
\end{proof}

We need two preparatory results for the proof of \ref{thm:supertope}. For the first one also see~\cite[Lemma 10]{HKK}. We reprove that result here, since in the presentation in~\cite{HKK} the signs are chosen the opposite way.  Recall from \ref{ex:cycle} that $\Cscr_n$ is the set of covectors of the OM of the directed cycle on $n$ vertices. 

\begin{Proposition}\label{lem:supertope}
Let $\M=(E,\mathcal{L})$, $\mathcal{L} \neq \mathcal{C}_n$, be a COM with tope set $\Tscr$ and let $R=\{+\}^E$. If for all $f\in E$ we have \new{$-_{f}R\in\Tscr$, then the poset $\Tscr_{R}$ is contractible.}
\end{Proposition}

\begin{proof}
We will show by induction that if $-R \notin \mathcal{T}$, then all covectors which contain exactly one plus-entry and at least one minus-entry are in $\mathcal{L}$. Since then in particular all covectors which contain exactly one minus-entry and one plus-entry (i.e. the cocircuits of $\mathcal{C}_n$) exist in $\mathcal{L}$, we get by (SE) that the all zero vector is in $\mathcal{L}$. Together, we can conclude that $\mathcal{C}_n \subseteq \mathcal{L}$, since we obtain all its covectors by composition of those vectors. Since $\mathcal{C}_n$ is uniform and contains $\mathbf{0}$, every other COM containing it has to be a OM which is free. This means that every possible sign-vector vector is a covector and is a consequence of the strong elimination axiom. Therefore $\mathcal{C}_n = \mathcal{L}$, since $-R \notin \mathcal{T}$. Taking the contraposition we see that if $\mathcal{L}\neq \Cscr_n$,
  then $-R\in\Tscr$, so the poset $\Tscr_R$ has a unique maximal element. In particular, it is contractible.

So let $-_{f}R \in \Tscr$ for all $f \in E$ and $-R\notin\Tscr$. We will use induction over the number of zero-entries in the covectors, i.e.\,we want to show that for every $n = 0,\dots,\lvert E \rvert-2$ all sign-vectors with $n$ zero entries, one plus-entry and $\lvert E \rvert-(n+1)$ minus-entries are in $\mathcal{L}$.  

$n=0$: By the existence of $-_{f}R$ there is nothing to show. We fix $n\ge0$ and assume that all covectors with $n$ or fewer zero-entries, exactly one plus-entry and at least one minus-entry exist in $\mathcal{L}$.

$n \rightarrow n+1 \leq \lvert E \rvert-2$: We will show that there
exists a covector with zero-entries in the i-th position, $i \in I
\subset E$, $\lvert I \rvert = n+1$, a plus-entry in the j-th
position, $j \notin I$ and $-$ everywhere else. We choose an $\hat{i}
\in I$ and consider two covectors with $0$ in $I \backslash \hat{i}$,
where one has a $+$ in the $\hat{i}$-th position and the other one
in the j-th position and both have a $-$ everywhere else. These do
exist by inductive assumption. W.l.o.g.  those two covectors look like
this:
\begin{align*}
&(0,\dots,0,\overbrace{+}^{\hat{i}},\;\,-\,\;,-,\dots,-)\\
&(\underbrace{0,\dots,0}_{I \backslash \hat{i}},\;\,-\,\;,\underbrace{+}_{j},-,\dots,-).
\end{align*}
 If we now perform strong elimination on $ \hat i$ with those two covectors we get the covector
\begin{align*}
X = (\underbrace{0,\dots,0}_{I \backslash \hat{i}},\underbrace{0}_{\hat{i}},\underbrace{*}_j,-,\dots,-).
\end{align*}
If $*$ was $-$, then $X \circ -_jR = \{-\}^{E}$. Since $\{-\}^{E}=-R\notin\Tscr$ we have $*=+$ and have the covector we were looking for, so $\mathcal{L} = \Cscr_n$. 

\end{proof}

\begin{Lemma} \label{lem:Cn} 
\new{Let $\M=(E,\mathcal{L})$ be a topal fiber of a COM $\M'=(E',\mathcal{L'})$, $R'\in \Tscr'$ a tope of $\M'$ and $R=R'_{|E}$ its restriction to $E$. If $\mathcal{L} = \mathcal{C}_n$, then $R \in \mathcal{L}$, in particular $R \neq \{+\}^E, \{-\}^E$.}

\end{Lemma}

\begin{proof}\new{Let $\M=(E,\mathcal{L})$ be a COM such that there is a COM $\M=(E',\mathcal{L'})$, with $E\subset E'$ and $\mathcal{L} = \rho_{(S^+,S^-)}(\Lscr')$ for some $S^+,S^- \subseteq E'$ and $\mathcal{L}=\mathcal{C}_n$.}
We saw in \ref{def:cycle} that $\mathbf{0} \in \mathcal{C}_n$. By the definition of $\rho_{(S^+,S^-)}(\Lscr')$ there exists $Z\in \mathcal{L}'$ with
$$Z_e = \begin{cases}
                        + & \text{ if } e \in S^+,\\
                        - & \text{ if } e \in S^-\\
                        0 & \text{ else. }
                       \end{cases} 
                $$
Since the composition of $Z$ with every other covector in $\mathcal{L}'$ is in $\mathcal{L}'$ (see \ref{rem:OMs}), we see that $\rho_{(S^+,S^-)}(\Lscr') = \mathcal{L}'\backslash \{S^+ \cup S^-\}$. So in this case $\mathcal{L}'\backslash \{S^+ \cup S^-\} = \mathcal{C}_n$, so every tope $R' \in \mathcal{L}'$ restricted to $E$ has to be in $\mathcal{C}_n$. Since $\{+\}^E, \{-\}^E \notin \mathcal{C}_n$, $R=R'|_E \neq \{+\}^E, \{-\}^E$.

\end{proof}

Now we are in position to prove \ref{thm:supertope}.

\begin{proof}[Proof of \ref{thm:supertope}] Let $\M=(E,\mathcal{L})$ be a COM such that there is a COM $\M'=(E',\mathcal{L'})$, with $E\subset E'$ and $\mathcal{L} = \rho_{(S^+,S^-)}(\Lscr')$ for some $S^+,S^- \subseteq E'$, $R'\in \Tscr'$ a tope of $\M'$ and $R=R'_{|E}$ its restriction to $E$. First we look at the case $\mathcal{L}=\mathcal{C}_n$. From \ref{lem:Cn} we know, $R$ is a tope of $\mathcal{C}_n$. So we have a unique minimal element and $\Tscr_R$ is contractible.
Now let $\mathcal{L} \neq \mathcal{C}_n$. Possibly reorienting
elements we may assume that $R=\{+\}^E$. We proceed by induction on
  $|E|$. If $|E|=1$ then $\Tscr_R$ either is a singleton or a chain of
  length 2 and thus contractible. Hence assume $|E| \ge 2$. If for
  all $f \in E$ there exists $-_{f}R$ as in \ref{lem:supertope}, then
  $\Tscr_R$ is contractible by \ref{lem:supertope}. Hence we may
  assume that there exists $f \in E$ such that $-_{f}R \not \in \Tscr$. Let $(\Tscr\hspace{-.1cm}\setminus \hspace{-.1cm}f)_{R \setminus f}$ denote the tope poset in $\Lscr \setminus f$ with base pattern $R \setminus f$. Since the class of COMs is closed under deletion, we know that $\Lscr'\setminus f$ is a COM. Since $\mathcal{L}$ evolved from $\mathcal{L}'$ by setting $\Lscr= \rho_{(S^+,S^-)}\Lscr'$ for some $S^+$ and $S^-$, $\Lscr \setminus f$ evolves in the same way from $\Lscr'\setminus f$, i.e. $\Lscr \setminus f = \rho_{(S^+,S^-)}(\Lscr'\setminus f)$ (note that $f$ cannot be in $S^+ \cup S^-$, since $E=E' \setminus (S^+ \cup S^-)$ and $f \in E$). Also $R \setminus f$ is the restriction to $E$ of the tope $R'\setminus f$ of $\Lscr'\setminus f$. We see that $\Lscr \setminus f$ together with $R \setminus f$ fulfills the assumptions of \ref{thm:supertope}. Furthermore, $\Lscr \setminus f \neq \mathcal{C}_n$, this follows from \ref{lem:Cn}. Hence,
  $(\Tscr\hspace{-.1cm}\setminus \hspace{-.1cm}f)_{R \setminus f}$ is contractible by inductive
  assumption.  
  We now want to show that $\Tscr_R$ and $(\Tscr\setminus f)_{R \setminus f}$ are homotopy equivalent by using \ref{lem:quillen}. So consider the order-preserving map $\pi_{f}:\Tscr_R \to (\Tscr\hspace{-.1cm}\setminus \hspace{-.1cm}f)_{R \setminus f}$ given by
  restriction. Let $Q \in (\Tscr\hspace{-.1cm}\setminus \hspace{-.1cm}f)_{R \setminus f}$. By \ref{lem:fiberisst}
  \[\pi_{f}^{-1}(((\Tscr\setminus f)_{R \setminus f})_{\preceq Q})=\Tscr(Q^+,\emptyset).\] 
  $\Tscr(Q^+,\emptyset)$ is the set of topes of $\rho_{(S^+\cup Q^+,S^-)}\Lscr'$. If $Q^+ \ne \emptyset$, then $\rho_{(S^+\cup Q^+,S^-)}\Lscr'$ has fewer elements than $\mathcal{L}$. Furthermore, by  \ref{lem:Cn}, $\rho_{(S^+\cup Q^+,S^-)}\Lscr' \neq \mathcal{C}_n$. Hence
  $\pi_{f}^{-1}((\Tscr\setminus f)_{\preceq Q})$ is contractible by inductive
  assumption. If $Q^+= \emptyset$ then by the choice of $f$ the
  preimage $\pi_{f}^{-1}(Q)$ is the all minus vector.
  Hence, this is the unique maximal element in
  $\pi_{f}^{-1}((\Tscr\setminus f)_{\preceq Q})$ and that fiber is also contractible. 
  So by \ref{lem:quillen} $\Tscr_R$ and $(\Tscr\setminus f)_{R \setminus f}$ are
  homotopy equivalent and the claim follows.
\end{proof}

We now introduce $\Tscr_{R,e}$, which is a truncated version of $\Tscr_R$. For $e \in E$ and $R \in \Tscr$ we define $\Tscr_{R,e}$ as the poset 
$\{ T \in \Tscr~|~T_e = -R_e\}\cup \{ \hat{0}\}$ 
with $\hat{0}$ as its artificial least element and the remaining poset 
structure induced from $\Tscr_R$.
For $P \in \Tscr_{R,e}$ we write $(\hat{0},P)_{R,e}$ for the interval from
$\hat{0}$ to $P$ in $\Tscr_{R,e}$. The following result will help us later in the main proof to obtain a factorization of the Varchenko matrix.

\begin{Lemma}
  \label{cor:crucial}
Let $\M=(E,\mathcal{L})$ be a COM, $R \in \Tscr$ a tope, $e\in E$ an element, $\hat{0} \neq P \in \Tscr_{R,e}$ and $S$ such that $e \not\in S\subseteq E$. Then 
\begin{align}\label{cruc1}
  \sum_{\genfrac{}{}{0pt}{}{Q \in \Tscr(\emptyset,\{e\})}{ S= S(P,Q) \cap S(Q,R)}}  \mu((\hat{0},Q)_{R,e}) = \left\{ \begin{array}{ccc} -1 & \mbox{~if~} & S = \emptyset \\ 
                                       0 & \mbox{~if~} & S \neq \emptyset 
                  \end{array} \right. 
\end{align}
and 
\begin{align}\label{cruc2}
  \sum_{\genfrac{}{}{0pt}{}{Q \in \Tscr(\{e\},\emptyset)}{ S= S(P,Q) \cap S(Q,R)}}  \mu((\hat{0},Q)_{R,e}) = \left\{ \begin{array}{ccc} -1 & \mbox{~if~} & S = \emptyset \\ 
                                       0 & \mbox{~if~} & S \neq \emptyset 
                  \end{array} \right. .
\end{align}
\end{Lemma}

\begin{proof}
In order to prove (\ref{cruc1}) we assume $R = \{+\}^E$. We prove the assertion by induction on $|S|$.
  If $S = \emptyset$ then 
  \begin{eqnarray*} 
     \sum_{\genfrac{}{}{0pt}{}{Q \in \Tscr(\emptyset,\{e\})}{ S= S(P,Q) \cap S(Q,R)}}  \mu((\hat{0},Q)_{R,e}) & = & \sum_{\hat{0} <_{R,e} Q \leq_{R,e} P} \mu((\hat{0},Q)_{R,e}) 
  \end{eqnarray*}
Note  that  $\sum_{\hat{0} <_{R,e} Q \leq_{R,e} P} \mu((\hat{0},Q)_{R,e}) = -\mu\left(\{Q \in \Tscr_{R,e}\,|\,\hat{0}<_{R,e} Q \leq_{R,e} P\}\right) -  \mu((\hat{0}, \hat{0})_{R,e})$. The poset $\{Q \in \Tscr_{R,e}\,|\,\hat{0}<_{R,e} Q \leq_{R,e} P\}$ has the maximal element $P$, so it is contractible and has M\"obius number $0$. Therefore we have
\begin{align*}
\sum_{\genfrac{}{}{0pt}{}{Q \in \Tscr(\emptyset,\{e\})}{ S= S(P,Q) \cap S(Q,R)}}  \mu((\hat{0},Q)_{R,e}) = -\mu((\hat{0}, \hat{0})_{R,e}) = -1.
\end{align*}
  Assume $|S| > 0$. 
  Set $$T^+ = \{ f\in E \setminus (S \cup \{e\})~|~P_f = +\}.$$
  
  Then 
  \begin{eqnarray} 
     \label{eq:supertope}
     \sum_{\genfrac{}{}{0pt}{}{Q \in \Tscr(\emptyset,\{e\})}{ S(P,Q) \cap S(Q,R) \subseteq S}}  \mu((\hat{0},Q)_{R,e}) & = & 
     \sum_{Q \in \Tscr(T^+,\{e\})} \mu((\hat{0},Q)_{R,e}). 
  \end{eqnarray}

  $\Tscr(T^+,\{e\})$ is the set of topes of the
  COM $\rho_{(T^+,\{e\})}\Lscr$, where the associated poset $\Tscr(T^+,\{e\})_R$ is contractible by \ref{thm:supertope} (note that we did not suppress the redundant coordinates $(T^+ \cup \{e\})$ here to keep notation simple).  Since the right
  hand side of \eqref{eq:supertope} ranges over the elements of a
  contractible poset, we have again
  \begin{align*}
  \sum_{Q \in \Tscr(T^+,\{e\})} \mu((\hat{0},Q)_{R,e}) = -\mu(\Tscr(T^+,\{e\})_R) -\mu(\hat{0},\hat{0}) = -\mu(\hat{0},\hat{0}) = -1.
  \end{align*}
  and showed

  \begin{eqnarray} 
     \label{eq:supertope2}
     \sum_{\genfrac{}{}{0pt}{}{Q \in \Tscr(\emptyset,\{e\})}{ S(P,Q) \cap S(Q,R) \subseteq S}}  \mu((\hat{0},Q)_{R,e}) & = & -1.
  \end{eqnarray}

  Now rewrite the left hand side of \eqref{eq:supertope2} as

  \begin{eqnarray} 
     \label{eq:supertope3}
     \sum_{\genfrac{}{}{0pt}{}{Q \in \Tscr(\emptyset,\{e\})}{ S(P,Q) \cap S(Q,R) \subseteq S}}  \mu((\hat{0},Q)_{R,e}) & = & \sum_{T \subseteq S} \sum_{\genfrac{}{}{0pt}{}{Q \in \Tscr(\emptyset,\{e\})}{ S(P,Q) \cap S(Q,R) = T}}  \mu((\hat{0},Q)_{R,e}) 
  \end{eqnarray}

  By induction the summand 
  $\displaystyle{\sum_{\genfrac{}{}{0pt}{}{Q \in \Tscr(\emptyset,\{e\})}{ S(P,Q) \cap S(Q,R) = T}}  \mu((\hat{0},Q)_{R,e})}$ 
  is $0$ for $T \neq S,\emptyset$ and $-1$ for $T = \emptyset$.
  Thus combining \eqref{eq:supertope2} and \eqref{eq:supertope3} we obtain:

  \begin{eqnarray*} 
     \label{eq:supertope4}
     -1 & = & \sum_{\genfrac{}{}{0pt}{}{Q \in \Tscr(\emptyset,\{e\})}{ S(P,Q) \cap S(Q,R) \subseteq S}}  \mu((\hat{0},Q)_{R,e}) \\
        & = & -1+\sum_{\genfrac{}{}{0pt}{}{Q \in \Tscr(\emptyset,\{e\})}{ S(P,Q) \cap S(Q,R) = S}}  \mu((\hat{0},Q)_{R,e}) 
  \end{eqnarray*}

  From this we conclude $$\sum_{\genfrac{}{}{0pt}{}{Q \in
      \Tscr(\emptyset,\{e\})}{ S(P,Q) \cap S(Q,R) = S}}
  \mu((\hat{0},Q)_{R,e}) = 0.$$ The second claim follows analogously
  by reorienting all the signs.
\end{proof}

We conclude this section with another result on contractability needed in the proof of our main theorem. We start with a lemma:

\begin{Lemma}\label{gateCOM} Let $\M=(E,\mathcal{L})$ be a COM, $X \in \mathcal{L}$ and $P \in \Tscr$. The tope $Q = X \circ P \in \cstar(X)$ is the only tope in $\cstar(X)$ such that for all $O \in \cstar(X)$ we have
  \begin{align}
   S(P,O) & =  S(P,Q) \cup S(Q,O) \label{eq:h1} \\  
    \emptyset & =   S(P,Q) \cap S(Q,O). \label{eq:h2}
  \end{align}
\end{Lemma}
\begin{proof}
 In the case where $X$ is a tope there is nothing to show, so we assume $z(X) \neq \emptyset$. It is easy to see that $Q$ fulfills (\ref{eq:h1}) and (\ref{eq:h2}).
  Let us assume there is another tope $Q^* \neq Q$ in $\cstar(X)$
  which has this property. By the definition of $Q$ we have
  \[S(Q,O)=S(P,O) \cap z(X)\text{ and }S(P,Q)=S(P,O) \backslash z(X)\text{ for all }O \in \cstar(X).\] 
  Since $Q^* \neq Q$ and $S(Q^*,O)$
  can only contain elements from $z(X)$, $S(P,Q^*)$ has to contain at least
  one element from $z(X)$. Now considering $O^* = (X \circ -Q^*) \in
  \cstar(X)$ we see that $S(P,Q^*) \cap S(Q^*,O^*) \neq \emptyset$, so
  $Q^*$ does not fulfill the property and we have a contradiction.
\end{proof}

For $e \in E$ and $P \in \Tscr$ we say that $e$ \emph{defines a proper face} of $P$
if there is a covector $X \in \mathcal{L}$ with $X \leq P$ and $X_e =
0$ with $X\neq \mathbf{0}$. Note that in this case there is a unique maximal such covector, namely the composition of all of them. Otherwise, we say that $e$ \emph{does not define a proper face} of $P$.

\begin{Theorem}
  \label{thm:moebius}
  Let $\M=(E,\mathcal{L})$ be a COM, $R \in \Tscr$ a tope, and let $e \in E$ define a proper face of $R$. 
  Let $Y \in \mathcal{L}$ 
  be the maximal covector such that $Y \leq R$ and $Y_e = 0$
  and choose 
  $P_{top} \in \Tscr_{R,e} \setminus \cstar(Y)$. Then
  $(\hat{0},P_{top})_{R,e}$ is contractible. 
  In particular, $\mu((\hat{0},P_{top})_{R,e}) = 0$.
\end{Theorem}

\begin{proof} 
Let $P \in (\hat{0},P_{top})_{R,e}$. Then by \ref{gateCOM} the tope
  $Q = Y \circ P \in \cstar(Y)$ is the unique tope in $\cstar(Y)$ 
  such that for all $O \in \cstar(Y)$ we have
  \begin{eqnarray*}
    S(P,O) & = & S(P,Q) \cup S(Q,O) \\
    \emptyset & = &  S(P,Q) \cap S(Q,O).
  \end{eqnarray*}

Since $Y_e = 0$ and $P \in \Tscr_{R,e}$ it also follows that $Q_e = -$.
  Since $Y \le R$, clearly $S(R,Q) =S(R,Y \circ P)\subseteq
  S(R,P)$ and hence $Q \preceq_R P$.  This shows $Q \in
  (\hat{0},P_{top})_{R,e}$.
We now define the map 
\begin{align*}
&\circ_Y: (\hat{0},P_{top})_{R,e}
  \rightarrow (\hat{0},P_{top})_{R,e}\\
&\circ_Y(P) = Y \circ P
\end{align*}  
   and prove that it is a closure operator by showing that it is order preserving and idempotent (i.e. $\circ_Y(\circ_Y(P)) = \circ_Y(P)$). So let $Q \preceq_R Q'$. Then $Y \circ Q
  \preceq_R Y \circ Q'$. Since $Y \leq R$ it follows that $Y \circ Q
  \preceq_R Q$. Obviously $Y \circ (Y \circ Q) = Y \circ Q$. So $\circ_Y$ is a closure operator and it follows that $(\hat{0},P_{top})_{R,e}$ is homotopy equivalent to its
  image (see e.g, \cite[Corollary 10.12]{Bj}).   

  Since $P_{top} \not\in \cstar(Y)$ and $Y \circ P_{top}  \in \cstar(Y)
  \cap (\hat{0},P_{top})_{R,e}$, 
  it also follows that $Y \circ Q \preceq_R Y \circ P_{top}$ for all 
  $Q \in (\hat{0},P_{top})_{R,e}$. Hence the image of $\circ_Y$ has
  a unique maximal element and hence is contractible.  
\end{proof}

\section{Main Proof}\label{sec:4}
\new{In this Section we assume that $\M=(E,\mathcal{L})$ is a COM with topes $\Tscr$ and signed Varchenko matrix $\Var$.} Recall that we assume $\Tscr$ to be linearly ordered. Note however that swapping two topes leads to a row swap and a column swap at the same time, so we do not change the sign of our determinant. Hence, in this section we will rearrange the ordering on $\Tscr$, whenever it is convenient for the proof. Moreover, for the proof we also fix a linear ordering on $E$, i.e., $E = \{ e_1 \prec \cdots \prec e_r\}$.

For any sign vector $\epsilon =
(\epsilon_1,\epsilon_2) \in \{ +,-\}^2$ let $\Var^{e,\epsilon}$
be a matrix with rows indexed by $\Tscr(\{e\} ,\emptyset)$ for $\epsilon_1 =+$,
$\Tscr(\emptyset, \{e\})$ for $\epsilon_1 = -$ and columns indexed
by $\Tscr(\{e\},\emptyset)$ for $\epsilon_2 = +$, $\Tscr(\emptyset, \{e\})$ for
$\epsilon_2 = -$. For a tope $R$ indexing a row and
a tope $Q$ indexing a column we set $\Var^{e,\epsilon}_{R,Q} = \Var_{R,Q}$. After reordering $\Tscr$ this yields a block decomposition of $\Var$ as

\begin{align}\label{equ:blocks}
\Var = \left( 
       \begin{array}{cc} 
         \Var^{e,(-,-)} & \Var^{e,(-,+)}    \\ 
         \Var^{e,(+,-)} & \Var^{e,(+,+)}  
       \end{array}
      \right) .
\end{align}

We fix such a linear ordering on $\Tscr$ and set $\mathcal{M}^e$ to 

\begin{align*}
\mathcal{M}^e_{Q,R} = \left\{ 
	    \begin{array}{ccc} 
		1 & \mbox{~if~} & Q = R \\ 
	       -\mu((\hat{0},Q)_{R,e})\, \Var_{Q,R} & \mbox{~if~} & e \text{ is the maximal element of }S(Q,R),  \\
		0 & \mbox{~otherwise} & 
	    \end{array} \right. . 
\end{align*}
	  
Note that this matrix has the following form

\begin{align*}
\mathcal{M}^e = 
      \left(  \begin{array}{cc} 
         \mathcal{I}^e_{\ell}  & U^e  \\ 
         L^e  & \mathcal{I}^e_{m} 
       \end{array}
      \right) ,
\end{align*}
      
where 
\begin{align*}
U^e_{Q,R} = -\mu((\hat{0},Q)_{R,e}) \, \Var_{Q,R},\;\;\;  &e \text{ is the maximal element of S}(Q,R), \\  &Q \in \Tscr(\emptyset,\{e\}), R \in \Tscr(\{e\},\emptyset) , \\
L^e_{Q,R} = -\mu((\hat{0},Q)_{R,e})\, \Var_{Q,R},\;\;\;  &e \text{ is the maximal element of S}(Q,R), \\ & Q \in \Tscr(\{e\},\emptyset), R \in \Tscr(\emptyset, \{e\})
\end{align*}
and $\mathcal{I}$ the identity matrix with $\ell=\#\Tscr(\emptyset,\{e\})$ and $m = \# \Tscr(\{e\},\emptyset)$.

\begin{Lemma}\label{lem:fac1}
   Let $e$ be the maximal element of $E$. Then 
   $\Var^{e,(-,+)}$ factors as
   \begin{align}\label{equ:minus}
      \Var^{e,(-,+)} = \Var^{e,(-,-)} \cdot U^e
   \end{align}
   and $\Var^{e,(+,-)}$ as
   \begin{align}\label{equ:plus}
      \Var^{e,(+,-)} = \Var^{e,(+,+)} \cdot L^e.
   \end{align}
\end{Lemma}
\begin{proof}
Let us prove (\ref{equ:minus}) first. For $P \in \Tscr(\emptyset,\{e\})$ and $R \in \Tscr(\{e\},\emptyset)$ the entry in row $P$ and
  column $R$ on the left hand side of (\ref{equ:minus}) is 
  $\Var_{P,R}$. On the right hand side the corresponding entry is:
  \begin{eqnarray*} 
     \sum_{Q \in \Tscr(\emptyset,\{e\})} \Var_{P,Q} \cdot U^e_{Q,R} & = & - \sum_{Q \in \Tscr(\emptyset,\{e\})} \mu((\hat{0},Q)_{R,e}) \cdot  \Var_{P,Q} \cdot \Var_{Q,R} 
  \end{eqnarray*} 

  This follows from the fact that $e$ is the maximal element of any separator of the topes indexing $U^e$. By definition we have for $Q \in \Tscr(\emptyset,\{e\})$ 
  \begin{eqnarray*}
    \Var_{P,Q} \cdot \Var_{Q,R} & = & \Var_{P,R} \cdot \prod_{f \in S(P,Q) \cap S(Q,R)} x_{f}^+ x_{f}^-.
  \end{eqnarray*}

We see that $ \Var_{P,Q} \cdot \Var_{Q,R}  =  \Var_{P,R}$ if $S(P,Q) \cap S(Q,R)=\emptyset$. Thus the claim of the lemma is proved once we have shown that for a fixed subset $S \subseteq E$ 
  and fixed $P,R$ we have:
  \begin{eqnarray}
     \label{eq:mueq}
     \sum_{\genfrac{}{}{0pt}{}{Q \in \Tscr(\emptyset,\{e\})}{ S= S(P,Q) \cap S(Q,R)}}  \mu((\hat{0},Q)_{R,e}) = \left\{ 
                \begin{array}{cc} 0 & \mbox{~if~} S\neq \emptyset \\
                                 -1 & \mbox{~otherwise.} 
                \end{array} \right. .
  \end{eqnarray}

  But this is the content of \ref{cor:crucial} and we are done. For (\ref{equ:plus}) the right hand side is
  \begin{align*}
  \sum_{Q \in \Tscr(\{e\},\emptyset)} \Var_{P,Q} \cdot L^e_{Q,R} & = & - \sum_{Q \in \Tscr(\{e\},\emptyset)} \mu((\hat{0},Q)_{R,e}) \cdot  \Var_{P,Q} \cdot \Var_{Q,R} 
  \end{align*}
and we can proceed analogous to the proof above.
\end{proof}

Next we use the matrices $\mathcal{M}^e$ to factorize $\Var$. 
The following lemma yields the base case for the inductive step in the 
factorization.

\begin{Lemma}
  \label{lem:fac}
  Let $e$ be the maximal element of $E$ and
  let $\Var_{x_e = 0}$ be the matrix $\Var$ after
  evaluating $x_{e}^+$ and $x_{e}^-$ to $0$. 
  Then 
  $$\Var = \Var_{x_e = 0} \cdot \mathcal{M}^e 
  $$
\end{Lemma}
\begin{proof}
Let $\Tscr$ be in that order, that we get the block decomposition (\ref{equ:blocks}) of $\Var$. Using lemma \ref{lem:fac1}, we see that
\begin{align}\label{eq:ma1}
\Var = \left( 
       \begin{array}{cc} 
         \Var^{e,(-,-)} & \Var^{e,(-,+)} \\ 
         \Var^{e,(+,-)} & \Var^{e,(+,+)} 
       \end{array}
      \right) & = \left( 
       \begin{array}{cc} 
         \Var^{e,(-,-)} & 0 \\ 
         0 & \Var^{e,(+,+)} 
       \end{array}
       \right) \cdot 
       \left( 
       \begin{array}{cc} 
         \mathcal{I}^e_{\ell}  & U^e  \\ 
         L^e  & \mathcal{I}^e_{m} 
       \end{array}
       \right) \\
      &=   \left( 
       \begin{array}{cc} 
	 \Var^{e,(-,-)} & 0 \\ 
	 0 & \Var^{e,(+,+)} 
       \end{array}
       \right) \cdot \mathcal{M}^e .
\end{align}

   Now the monomial $\Var_{P,Q}$ has a factor $x_{e}^+$ or $x_{e}^-$ if and only if
   $P \in \Tscr(\emptyset,\{e\})$ and $Q \in \Tscr(\{e\},\emptyset)$ or
   $P \in \Tscr(\{e\},\emptyset)$ and $Q \in \Tscr(\emptyset,\{e\})$.
   Hence 
   \begin{eqnarray}
     \label{eq:ma2}
     \Var_{x_e = 0}  & = & \left( 
        \begin{array}{cc} 
           \Var^{e,(-,-)} & 0 \\ 
	   0 & \Var^{e,(+,+)} 
	\end{array} 
	\right).
   \end{eqnarray}
   Combining \eqref{eq:ma1} and \eqref{eq:ma2} yields the claim.
\end{proof}

Now we are in position to state and prove the crucial factorization.

\begin{Proposition}
  \label{pro:fac}
  Let $E = \{ e_1 \prec \cdots \prec e_r\}$ be a fixed
  ordering. Then 
  $$\Var = \mathcal{M}^{e_1} \cdots \mathcal{M}^{e_r}.$$ 
\end{Proposition}
\begin{proof}
  We will prove by downward induction on $i$ that
  \begin{eqnarray}
     \label{eq:var}
     \Var & = & \Var_{x_i=\cdots = x_r = 0} \cdot \mathcal{M}^{e_i} \cdots \mathcal{M}^{e_r}.
  \end{eqnarray}

  For $i=r$ the assertion follows directly from 
  \ref{lem:fac}.
  For the inductive step assume
  $i > 1$ and \eqref{eq:var} holds for $i$. 
  We know from \ref{lem:fac} that if we choose a linear ordering
  on $E$ for which $e_{i-1}$ is the largest element then 
  \begin{eqnarray}
	    \label{eq:ind1} 
	    \Var & = & \Var_{x_{i-1}=0}\cdot \mathcal{N},
	  \end{eqnarray}
	  where
	  $\mathcal{N} = (N_{Q,R})_{Q,R \in \Tscr}$ is defined as
	  $$N_{Q,R} = \left\{ 
	    \begin{array}{ccc} 
		1 & \mbox{~if~} & Q = R \\ 
	       -\mu((\hat{0},Q)_{R,e_{i-1}})\, \Var_{Q,R} & \mbox{~if~} & e_{i-1} \in \text{Sep}(Q,R)\\

		0 & \mbox{~otherwise} & 
	    \end{array} \right. , 
	  $$
	  Since $\mathcal{N} = \mathcal{M}^{e_{i-1}}$ for this particular ordering. Now we go back to the ordering in the assumption and set $x_i=\cdots   = x_r = 0$ in $\mathcal{N}$. We see that
	  $$(N_{Q,R})_{x_i=\cdots   = x_r = 0} \left\{ 
	    { \begin{array}{cl} 
		1 & \mbox{~if~}  Q = R \\ 
	       -\mu((\hat{0},Q)_{R,e_{i-1}})\, \Var_{Q,R} & \mbox{~if~} e_{i-1} \mbox{~is the largest element in~} S(Q,R) \\
		0 & \mbox{~otherwise}  
	    \end{array}} \right. . 
	  $$
	  But then
	  $\mathcal{N}_{x_i=\cdots   = x_r = 0} = \mathcal{M}^{e_{i-1}}$. 
	 
	  Now \eqref{eq:ind1} implies
	  \begin{eqnarray*}
	      \Var_{x_i=\cdots = x_{r}=0} &= & \Var_{x_{i-1}= \cdots = x_{r} = 0} \cdot
	       \mathcal{N}_{x_i = \cdots = x_r = 0} \\
					  & = & \Var_{x_{i-1}= \cdots = x_{r} = 0} \cdot
	\mathcal{M}^{e_{i-1}} 
	  \end{eqnarray*}

	  With the induction hypothesis this completes the induction step by
	  \begin{eqnarray*}
	     \Var & = & \Var_{x_i=\cdots = x_r = 0} \cdot \mathcal{M}^{e_i} \cdots \mathcal{M}^{e_r} \\
		  & = & \Var_{x_{i-1} = x_i=\cdots = x_r = 0} \cdot \mathcal{M}^{e_{i-1}} \cdots \mathcal{M}^{e_r}. 
	  \end{eqnarray*}

	  For $i =1$ the matrix $\Var_{x_1=\cdots = x_r = 0}$ is the identity matrix. 
	  Thus \eqref{eq:var} yields:
	  $$\Var = \mathcal{M}^{e_1} \cdots \mathcal{M}^{e_r}.$$
	\end{proof}
Before we prove the following proposition, we quote~\cite[Corollary 3]{WH}, which is a result for oriented matroids.

\begin{Lemma}\label{lem:moebius}
Let $\mathbf{0} \in \mathcal{L}$ and let $P \in \Tscr_{R,e}$ such that $e$ does not define a proper face of $P$.
  Then the M\"obius 
  number 
  $\mu((\hat{0},P)_{R,e})$ is $0$ if $-R \neq P$ and   
  $(-1)^{\rank(\mathcal{L})}$ if $-R = P$. 
\end{Lemma}

Now let $Y \in \mathcal{L}$ and $e \in z(Y)$ be the maximal element of
$z(Y)$. Define $\Tscr^{Y,e}$ as the set of topes $P \in \Tscr$
such that $Y$ is the maximal element of $\mathcal{L}$ for which $Y_e = 0$ and $Y < P$.
	 
\begin{Proposition}\label{pro:contraction}
  For any pair of topes  $Q,R \in \Tscr^{Y,e}$ we have 
  \begin{eqnarray*} 
     \mu((\hat{0},Q)_{R,e}) = 
       \left\{ \begin{array}{ccc} 
          (-1)^{\rank(\mathcal{L}|_{z(Y)})} & \mbox{~if~} & Q_{z(Y)} = -R_{z(Y)} \\ 
          0 & \mbox{~otherwise~} & 
  \end{array} \right. .  
  \end{eqnarray*} 
  
\end{Proposition}

\begin{proof}
  By the definition of $\Tscr^{Y,e}$ we have $Y \le Q,R$, so $Y_e = Q_e = R_e$ for all $e \in E\backslash z(Y)$. Thus, if we consider the poset $\Tscr_{R|_{z(Y)},e}$ in the restriction $\mathcal{L}|_{z(Y)}$ we find that the interval $(\hat{0},Q)_{R,e}$ is isomorphic to $(\hat{0}, Q|_{z(Y)})_{ R|_{z(Y)},e}$, since the elements in $(\hat{0},Q)_{R,e}$ only differ in $z(Y)$. We saw in  \ref{sec:2}, that $(E\backslash \underline{Y}, F(Y)\backslash \underline{Y})$ is an OM. Further, $\Tscr_{R|_{z(Y)},e}$ is a poset and $(\hat{0}, Q|_{z(Y)})_{ R|_{z(Y)},e}$ is an interval in this particular OM. Furthermore, since $Y$
is the maximal element satisfying $Y_e=0$ and $Y \le Q$, $e$ does not define a proper face of $Q|_{z(Y)}$. Since our interval is in an OM, we can use \ref{lem:moebius} and the claim follows.
\end{proof}
We define $b_{Y,e} = 0$ if $e$ is not the maximal element of 
$z(Y)$ and $\frac{1}{2} \# \Tscr^{Y,e}$ otherwise. 
Since $P$ together with $Y \circ (-P)$ is a perfect pairing on 
$\Tscr^{Y,e}$ it follows that $\Tscr^{Y,e}$ contains an even number of
topes. In particular, $b_{Y,e}$ is a nonnegative integer. 
We denote by $\mathcal{M}^{Y,e}$ the submatrix of $\mathcal{M}^e$ obtained by selecting rows and columns indexed by $\Tscr^{Y,e}$.  

\begin{Lemma}
  \label{lem:formula}
  Let $Y \in \mathcal{L}$ and $e \in z(Y)$.
  If $\Tscr^{Y,e} \neq \emptyset$. 
  then 
  $$\det(\mathcal{M}^{Y,e}) = (1-a(Y))^{b_{Y,e}}$$
  where $a(Y) := \prod_{e \in z(Y)} x_{e}^+ x_{e}^-$.
\end{Lemma}
\begin{proof}
If $Q_{z(Y)} = -R_{z(Y)}$ then $\Var_{Q,R} = \prod_{e \in z(Y), Q_e = *} x_{e}^*$.
  Using the definition of $\mathcal{M}^e$ and ~\ref{pro:contraction} we find
  \begin{eqnarray*} 
    \mathcal{M}^{Y,e}_{Q,R} = 
       \left\{ 
          \begin{array}{ccc} 
                             1 & \mbox{~if~} & Q = R \\ 
             -(-1)^{\rank(\mathcal{L}|_{z(Y)})} \prod_{e \in z(Y), Q_e = *} x_{e}^* & \mbox{~if~} & Q = Y \circ (-R) \\
                                       &             & e \mbox{ largest element of } S(Q,R) \\
                             0 & \mbox{~otherwise~} & 
          \end{array} \right. .
  \end{eqnarray*} 
  We order rows and columns of $\mathcal{M}^{Y,e}$ so that the elements 
  $R$ and $Y \circ (-R)$ are  
  paired in consecutive rows and columns.
  With this ordering $\mathcal{M}^{Y,e}$ is a block diagonal
  matrix having along its diagonal $b_{Y,e}$ two by two matrices 
  $$\left( \begin{array}{cc}
        1                         & -(-1)^{\rank(\mathcal{L}|_{z(Y)})} \prod_{e \in z(Y), R_e = *} x_{e}^* \\
        -(-1)^{\rank(\mathcal{L}|_{z(Y)})} \prod_{e \in z(Y), -R_e = *} x_{e}^* & 1 
           \end{array} \right)
  $$
  if $e$ is the maximal element of $z(Y)$ and 
  identity matrices otherwise.  In any case we find
  $\det(\mathcal{M}^{Y,e}) = (1-a(Y))^{b_{Y,e}}$ as desired.
\end{proof}

\begin{Lemma}
  \label{lem:block}
  After suitably ordering $\Tscr$ the matrix $\mathcal{M}^e$ is the block 
  lower triangular matrix
  with the matrices $\mathcal{M}^{Y,e}$ for $Y \in \mathcal{L}$ with $Y_e = 0$ and
  $\Tscr^{Y,e} \neq \emptyset$ on the main
  diagonal.
\end{Lemma}
\begin{proof}
Note that for each tope $T$ there is exactly one $Y$ with $T \in \mathcal{T}^{Y,e}$, and
that $\mathcal{T}^{0,e} = \emptyset$. We fix a linear ordering of $\Tscr$ such that for each $Y \in \mathcal{L}$
the topes from $\Tscr^{Y,e}$ form an interval and such that the topes from $\Tscr^{Y,e}$
precede those of $\Tscr^{Y',e}$ if $Y < Y'$.

  For this order the claim follows if we show that the entry $(\mathcal{M}^e)_{Q,R}$ is zero 
  whenever $Q \in \Tscr^{Y',e}$, $R \in \Tscr^{Y,e}$ and $Y' < Y$.

  If $Q_e = R_e$ then by $Q \neq R$ we have 
  $(\mathcal{M}^e)_{Q,R} = 0$. Hence it suffices to consider the case $Q_e \neq R_e$. Since $Y \neq \mathbf{0}$, $e$ is a proper face of $R$.
 
 If $Q \not\in \cstar(Y),\, Q \in \Tscr(\emptyset,\{e\})$ and $R\in \Tscr(\{e\},\emptyset)$ then 
  it follows from \ref{thm:moebius} that 
  $\mu((\hat{0},Q)_{R,e}) = 0$ and therefore $(\mathcal{M}^e)_{Q,R} = 0$. 
  Analogously if $Q \not\in \cstar(Y),\,Q \in \Tscr(\{e\},\emptyset)$ and $R\in \Tscr(\emptyset,\{e\})$ then 
  $\mu((\hat{0},Q)_{R,e}) = 0$ and therefore $(\mathcal{M}^e)_{Q,R} = 0$. 

  On the other hand, if $Q \in \cstar(Y)$, then in particular $Y \leq
  Q$. Since by definition of $\Tscr^{Y',e}$ we have that $Y'$ is the
  maximal covector such that $Y' \leq Q$ and $Y'_e=0$ it follows that
  $Y \leq Y'$.  Since $Y \neq Y'$ we must have that $Y < Y'$, i.e.\
  $(\mathcal{M}^e)_{Q,R}$ is an entry below the diagonal and we are done.

\end{proof}

\begin{proof}[Proof of \ref{thm:varchenko}]
  After fixing a linear order on $E$ it follows from \ref{pro:fac} 
  that $\det \Var$ is 
  the product of the determinants of $\mathcal{M}^e$ for $e \in E$.  
  By \ref{lem:block} the determinant of each $\mathcal{M}^e$ is
  a product of determinants of $\mathcal{M}^{Y,e}$ for $e \in E$ and $Y \in \mathcal{L}$
  for which $\Tscr^{Y,e} \neq \emptyset$. Then \ref{lem:formula} completes
  the proof.
\end{proof}

\begin{Remark}[Description of $b_Y$]\label{note:bY}
In \ref{thm:varchenko} we describe $b_Y$ as a nonnegative integer, but this can be made more precise: Fix any linear order on $E$ and let $e_Y$ be the maximal element of $z(Y)$. From \ref{lem:formula} we deduce, that $b_Y = b_{Y,e_Y}$. Thus $2b_Y
$ counts the topes $P \in \Tscr$
such that $Y$ is the maximal element of $\mathcal{L}$ for which $Y_{e_Y} = 0$ and $Y \leq P$. In particular, $b_Y$ does not depend on the choice of the linear ordering on $E$.
\end{Remark}

\section{Applications}\label{sec:applications}
We give two applications of our formula for the Varchenko determinant on two COMs associated to a poset $\mathcal{P}$: its lattice of ideals and its set of linear extensions. As an example we will use the poset $\mathcal{Q}$ in Figure~\ref{fig:poset}.
 \begin{center}
 \begin{figure}
 \includegraphics[width=\textwidth]{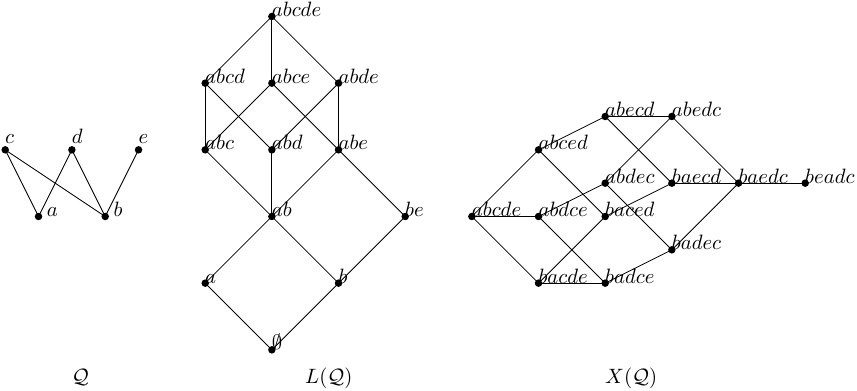}
 \caption{A poset $\mathcal{Q}$, its lattice $L(\mathcal{Q})$ of ideals and its set $X(\mathcal{Q})$ of linear extensions. Edges in the graphs in the middle and on the right are drawn if endpoints correspond to topes with separator consisting of a single element. Edges corresponding to the same element are parallel.}\label{fig:poset}
  
 \end{figure}

 \end{center}

\subsection{Distributive Lattices}
By the Fundamental Theorem of Finite Distributive Lattices, for every distributive lattice $L$ there exists a poset $\mathcal{P}$, such that ordering the ideals (downward closed sets) of $\mathcal{P}$ by inclusion yields a lattice isomorphic to $L$. The topes of the COM associated to $L$ correspond to the ideals of $\mathcal{P}$, the empty set can be seen as the all-plus vector, the ground set $E$ of this COM is the ground set of $\mathcal{P}$, and the separator of two ideals $I,I'$ is the symmetric difference $I\Delta I'$. So this allows, to quickly write down the (unsigned) Varchenko matrix $\mathbf{V}_L$ of $L$. In our example we indicate $\mathbf{V}_L(\mathcal{P})$ in the following way, where we just display the elements of the symmetric difference of two ideals to make it easier to read. Note that in order to get the Varchenko matrix itself one has to exchange a string $s_1 \ldots s_k$ for the product $\prod_{s\in S}x_s$. The $\emptyset$ translates therefore to the empty product, which is 1.

$$
\begin{pmatrix}
 \emptyset & a & b & ab & be & abc & abd & abe & abcd & abce & abde & abcde\\
 a & \emptyset & ab &b & abe& bc & bd & be & bcd & bce & bde & bcde &\\
 b& ab & \emptyset & a & e & ac & ad & ae & acd & ace & ade & acde \\
 ab & b & a & \emptyset & ae & c & d &  e & cd & ce & de & cde \\
 be & abe & e & ae & \emptyset & ace & ade & a & acde & ac & ad & acd \\
 abc & bc & ac & c & ace  & \emptyset & cd & ce & d & e & cde & de \\
 abd & bd & ad & d & ade & cd & \emptyset & de & c & cde & e & ce \\
 abe & be & ae & e & a & ce & de & \emptyset & cde & c & d & cd \\
 abcd & bcd & acd & cd & acde & d & c & cde & \emptyset & de & ce & e \\
 abce & bce & ace & ce & ac & e & ced & c & de & \emptyset & cd & d \\
 abde & bde & ade & de & ad & cde & e & d & ce & cd & \emptyset & c \\
 abcde & bcde & acde & cde & acd & de & ce & cd & e & d & c & \emptyset 
\end{pmatrix}$$

Let us define the covectors of that COM: Let $I \subseteq I'$ be two ideals such that $I'\backslash I$ forms an antichain. Then these two ideals define a covector $Y$ by setting
\begin{align*}
    Y(I,I')_e = \begin{cases}
        - \text{ if }e \in I\\
        0 \text{ if }e \in I'\backslash I\\
        + \text{ otherwise. }\\
    \end{cases}
\end{align*}

In particular, when $I=I'$, we get a tope corresponding to the ideal $I$ and the all-plus tope corresponds to the empty ideal.

Now, if we pick a linear ordering on $E$, let $e_Y$ be the largest element of $I'\setminus I$, then $2b_{Y(I,I')}$ counts those ideals $K$ such that we have 
\begin{itemize}
 \item $I\subseteq K\subseteq I'$,
 \item if $J\subseteq K\subseteq J'$ and $e_Y$ is the largest element of $J'\setminus J$, then $I'\setminus I \subsetneq J'\setminus J$.
\end{itemize}
But note that this condition is only satisfied if $I=K$ and $I'=I\cup \{e\} $ for some $e\in\mathcal{P}$ or $I'=K$ and $I=I'\setminus \{e\} $ for some $e\in\mathcal{P}$. Indeed, if otherwise $I\subsetneq K \subsetneq I'$ and $e_Y$ is the largest element of $I'\setminus I$ one can set $J = K \setminus \{e_Y\}$ and $J'=K \cup \{e_Y\}$, a contradiction to the above condition. 
Hence $b_{Y(I,I')}$ is $1$ if $|I'\setminus I|=1$ and $0$ otherwise. All pairs $I,I'$ with $|I'\setminus I|=1$ look like $I = I'\backslash \{p\}$, where $p$ is a maximal element of $I'$.

Thus,~\ref{cor:varchenko} and~\ref{note:bY} yield that 
  \begin{align*}
     \det (\mathbf{V}_L) = \prod_{I \in \mathcal{I}} \prod_{p\in \max(I)}(1-x_p^2)=\prod_{p\in \mathcal{P}}(1-x_p^2)^{m_p},
   \end{align*}
where $\mathcal{I}$ denotes the set of ideals of $\mathcal{P}$, $\max(I)$ the set of maximal elements of an ideal $I$ and $m_p$ denotes the number of ideals having $p$ as maximal element.
In our example we get the following formula for $\det (\mathbf{V}_{L(\mathcal{Q})})$: 

\noindent$(1-x_a^2) \cdot(1-x_b^2)\cdot((1-x_a^2)(1-x_b^2))\cdot(1-x_e^2)\cdot(1-x_c^2)\cdot(1-x_d^2)\cdot((1-x_a^2)(1-x_e^2))\cdot((1-x_c^2)(1-x_d^2))\cdot((1-x_c^2)(1-x_e^2))\cdot((1-x_d^2)(1-x_e^2))\cdot((1-x_c^2)(1-x_d^2)(1-x_e^2))=$
\[(1-x_a^2)^3 (1-x_b^2)^2  (1-x_c^2)^4 (1-x_d^2)^4 (1-x_e^2)^5.\]

\subsection{Linear extensions}

Another instance is the ranking COM of a poset $\mathcal{P}$, that was described in~\cite{BCK}. The topes are the linear extensions of $\mathcal{P}$, and the separator of two linear extensions $L, L'$ is the set of pairs of elements of $\mathcal{P}$ that are ordered differently in $L$ and $L'$. In particular, the ground set of this COM consists of the set $\mathrm{Inc}(\mathcal{P})$ of incomparable pairs of $\mathcal{P}$, e.g., $\mathrm{Inc}(\mathcal{Q})=\{ab, ae, cd, ce, de\}$. We can thus define the (unsigned) Varchenko matrix $\mathbf{V}_{X(\mathcal{P})}$. 
We get a description of $\mathbf{V}_{X(\mathcal{Q})}$. We deem it too large to display it entirely, but for example the entry corresponding to extensions $abcde, beadc$ is $x_{ab}x_{ae}x_{cd}x_{ce}x_{de}$.

The covectors of the ranking COM are the weak extensions of $\mathcal{P}$, i.e., those poset extensions of $\mathcal{P}$ that are chains of antichains. The set $z(Y)$ of such an extension $Y$ corresponds to its set of incomparable pairs $\mathrm{Inc}(Y)$. In order to properly define the signs of the covectors, one can pick an arbitrary linear extension $L_0$ of $\mathcal{P}$, and set non-zero coordinates of $Y$ to $+$ if the corresponding incomparable pair of $\mathcal{P}$ is ordered the same way in $L_0$ and $Y$ and to $-$ otherwise. To define $b_Y$ we can fix an arbitrary linear order on the set $\mathrm{Inc}(\mathcal{P})$ and let $e_Y=\{p,q\}$ be the largest element of $\mathrm{Inc}(Y)$. Then $2b_{Y}$ counts linear extensions $L$ of $\mathcal{P}$ such that
\begin{itemize}
 \item $L$ is a linear extension of $Y$,
 \item if another weak extension $Z$ of $\mathcal{P}$ has $e_Y$ as largest incomparable pair, then either $L$ is not an extension of $Z$ or $Z$ is not an extension of $Y$.
\end{itemize}
In this setting one can see that no such $Z$ can exist if and only if $Y$ is a chain of antichains only one of which - say $A$ - has size larger than $1$. In this case the feasible $L$ are extensions of $Y$ that extend $A$ by starting and ending with an element among $\{p,q\}$. Hence, there are $2(|A|-2)!$ such linear extensions. By~\ref{cor:varchenko} and~\ref{note:bY} we have 
  \begin{align*}
     \det (\mathbf{V}_{\mathcal{P}}) = \prod_{A \in \mathcal{A}_{\geq 2}} (1-\prod_{p\neq q\in A}x_{p,q}^2)^{(|A|-2)!},
   \end{align*}
where $\mathcal{A}_{\geq 2}$ denotes the set of antichains of size at least $2$ of $\mathcal{P}$.

\section{Conclusion}\label{sec:conclusions}
One might wonder to what extent our result could be further generalized to other classes. A natural next class are partial cubes, i.e., isometric subgraphs of the hypercube $Q_d$. These generalize (tope graphs of) COMs and allow for an analogous definition of the Varchenko matrix, where the $(u,v)$ entry contains a product of monomials indexed by those coordinates in $\{1,\ldots, d\}$ where $u$ and $v$ differ. The smallest partial cube that is not the tope graph of a COM is the full subdivision of $K_4$, see~\cite{KM}. In this case the Varchenko matrix looks like the following
\begin{align*}
\scriptsize
\left(\begin{array}{cccccccccc}
 1&x_1&x_2&x_3&x_1 x_4&x_1 x_3 x_4&x_3 x_4&x_2 x_3 x_4&x_2 x_4&x_1 x_2 x_4\\
 x_1&1&x_1 x_2&x_1 x_3&x_4&x_3 x_4&x_1 x_3 x_4&x_1 x_2 x_3 x_4&x_1 x_2 x_4&x_2 x_4\\ 
x_2&x_1 x_2&1&x_2 x_3&x_1 x_2 x_4&x_1 x_2 x_3 x_4&x_2 x_3 x_4&x_3 x_4&x_4&x_1 x_4\\
x_3&x_1 x_3&x_2 x_3&1&x_1 x_3 x_4&x_1 x_4&x_4&x_2 x_4&x_2 x_3 x_4&x_1 x_2 x_3 x_4\\ x_1 x_4&x_4&x_1 x_2 x_4&x_1 x_3 x_4&1&x_3&x_1 x_3&x_1 x_2 x_3&x_1 x_2&x_2\\
 x_1 x_3 x_4&x_3 x_4&x_1 x_2 x_3 x_4&x_1 x_4&x_3&1&x_1&x_1 x_2&x_1 x_2 x_3&x_2 x_3\\
x_3 x_4&x_1 x_3 x_4&x_2 x_3 x_4&x_4&x_1 x_3&x_1&1&x_2&x_2 x_3&x_1 x_2 x_3\\ x_2 x_3 x_4&x_1 x_2 x_3 x_4&x_3 x_4&x_2 x_4&x_1 x_2 x_3&x_1 x_2&x_2&1&x_3&x_1 x_3\\
x_2 x_4&x_1 x_2 x_4&x_4&x_2 x_3 x_4&x_1 x_2&x_1 x_2 x_3&x_2 x_3&x_3&1&x_1\\
 x_1 x_2 x_4&x_2 x_4&x_1 x_4&x_1 x_2 x_3 x_4&x_2&x_2 x_3&x_1 x_2 x_3&x_1 x_3&x_1&1
\end{array}\right)
\end{align*}

and its determinant is of the following form:
\begin{align*}
&(x_4 - 1)^3 (x_4 + 1)^3 (x_3 - 1)^3 (x_3 + 1)^3 (x_2 - 1)^3 (x_2 + 1)^3 (x_1 - 1)^3 (x_1 + 1)^3 \\ 
&(3 x_1^2 x_2^2 x_3^2 x_4^2 - x_1^2 x_2^2 x_3^2 - x_1^2 x_2^2 x_4^2 - x_1^2 x_3^2 x_4^2 - x_2^2 x_3^2 x_4^2 + 1)
\end{align*}
 Thus, in this case there is no nice factorization. 
\begin{Problem}
Are there classes of partial cubes beyond COMs, that allow for a factorization theorem of the Varchenko matrix? 
\end{Problem}

As mentioned in the introduction, we are not aware of
  an example of a COM which cannot be extended to become the supertope
  of an oriented matroid. The conjectures 
  from~\cite[Conjecture 1]{BCK} and~\cite[Conjecture 1]{KM} in our
  language are equivalent to the following:

\begin{Problem}\label{prob:supertopes}
  Are supertopes of oriented matroids a proper subclass of the class of complexes of oriented matroids?
\end{Problem}

\subsubsection*{Acknowledgements:} We sincerely thank the referees for their attentive review and valuable comments, which have significantly improved the quality of this work. KK was partially supported by the French \emph{Agence nationale de la recherche} through project ANR-17-CE40-0015 and by the Spanish \emph{Ministerio de Econom\'ia, Industria y Competitividad}
through grant RYC-2017-22701, grant PID2019-104844GB-I00 and grant PID2022-137283NB-C22.

%


\begin{thebibliography}{xxx}
  \bibitem{AM}
      M. Aguiar, S. Mahajan, 
      Topics in hyperplane arrangements, 
      Mathematical Surveys and Monographs {\bf 226},
      American Mathematical Society, Providence, RI, 2017. 
  \bibitem{BCK}
     H.-J. Bandelt, V. Chepoi , K. Knauer,  
     COMs: complexes of oriented matroids,
     J. Combin. Theory Ser. A, {\bf 156} (2018) 195--237. 
  \bibitem{Bj}
     A. Bj\"orner, Topological methods, in: Handbook of combinatorics, Volume 2, 1819--1872, Elsevier Sci. B. V., Amsterdam, 1995. 
  \bibitem{thebook}
      A. Bj\"orner, M. Las Vergnas, B. Sturmfels, W. White, G.M. Ziegler,
     Oriented matroids. Second edition. Encyclopedia of Mathematics and its
     Applications {\bf 46}. Cambridge University Press, Cambridge, 1999.
  \bibitem{BV}
     T. Bry\l awski, A. Varchenko, 
     The determinant formula for a matroid bilinear form, 
     Adv. Math. {\bf 129} (1997) 1--24. 
  \bibitem{DH}
      G. Denham, P. Hanlon,
      Some algebraic properties of the Schechtman-Varchenko bilinear forms, 
      in: New perspectives in algebraic combinatorics (Berkeley, CA, 1996–97), 
      149--176, Math. Sci. Res. Inst. Publ., {\bf 38}, Cambridge University Press, Cambridge, 
      1999.
  \bibitem{G} 
      R. Gente, 
      The Varchenko Matrix for Cones, 
      PhD-Thesis, Philipps-Universit\"at Marburg, 2013.
%
  \bibitem{HKK}
  	W. Hochstättler, S. Keip, and K.Knauer. "Kirchberger's theorem for complexes of oriented matroids." Linear Algebra and its Applications 693 (2024): 288-296.    
  
  \bibitem{WH}
  	W. Hochst\"attler, V. Welker. "The Varchenko determinant for oriented matroids." Mathematische Zeitschrift 293.3 (2019): 1415-1430.    
  
  \bibitem{KM}
     K. Knauer,  T. Marc,
     On tope graphs of complexes of oriented matroids,
     Discrete Comput. Geom. 63 (2020): 377-417.
   \bibitem{Q}  
 D. Quillen,  "Homotopy properties of the poset of nontrivial p-subgroups of a group." Advances in Mathematics 28.2 (1978): 101-128. 
  \bibitem{R1}
      H.\ Randriamaro, 
      The Varchenko Determinant of an Oriented  Matroid,
      Trans.\ Comb.\ (10) {\bf 4} (2021), 7--18.
      \bibitem{R2}
H. Randriamaro. "The Varchenko matrix for topoplane arrangements." Communications in Contemporary Mathematics 24.10 (2022): 2150086.
  \bibitem{R3}
 H. 	Randriamaro, Computer Algebra of Conditional Oriented Matroids, Habilitation, (2023)
Universität Kassel.    
  \bibitem{SV1} 
      V.V. Schechtman, A.N. Varchenko, 
      Arrangements of hyperplanes and Lie algebra homology,
      Invent. Math. {\bf 106} (1991) 139--194. 
  \bibitem{St}
     R.P. Stanley,
      Enumerative combinatorics. Volume 1. 
      Second edition. Cambridge Studies in Advanced Mathematics {\bf 49},
  \bibitem{V} 
     A. Varchenko, 
     Bilinear form of real configuration of hyperplanes,
     Adv. Math. {\bf 97} (1993) 110--144. 
  \bibitem{W}
	M. L. Wachs, Poset topology: tools and applications.  (2006).
\end{thebibliography}
\end{document}